\documentclass{article}
\usepackage{amssymb}
\usepackage{color}
\usepackage{amsthm}
\usepackage{mhsetup}
\usepackage{mathtools}
\usepackage{tcolorbox}

\mathtoolsset{showonlyrefs} 
\numberwithin{equation}{section}

\newcommand{\RR}{\mathbb{R}}
\newcommand{\NN}{\mathbb{N}}
\newcommand{\ZZ}{\mathbb{Z}}
\newcommand{\dist}{\operatorname{dist}}
\newcommand{\cat}{\operatorname{cat}}
\newcommand{\dH}{\operatorname{d_{H}}}
\newtheorem{Theorem}{Theorem}[section]

\newtheorem{Lemma}[Theorem]{Lemma}
\theoremstyle{definition}
\newtheorem{Definition}[Theorem]{Definition}
\newtheorem{Assumption}[Theorem]{Assumption}
\theoremstyle{remark}
\newtheorem{Remark}[Theorem]{Remark}

\begin{document}
\title{Infinitely many periodic solutions to a Lorentz force equation with singular electromagnetic potential\thanks{Under the auspices of INdAM-GNAMPA, Italy.}
}
\author{A. Boscaggin, W. Dambrosio and D. Papini}
\date{}
\maketitle
\begin{abstract}
We consider the Lorentz force equation
$$
\frac{d}{dt}\left(\frac{m\dot{x}}{\sqrt{1-|\dot{x}|^{2}/c^{2}}}\right) =
q \left(E(t,x) + \dot x \times B(t,x)\right), \qquad x \in \mathbb{R}^3,
$$
in the physically relevant case of a singular electric field $E$. Assuming that $E$ and $B$ are $T$-periodic in time and satisfy suitable further conditions, we prove the existence of infinitely many $T$-periodic solutions. The proof is based on a min-max principle of Lusternik-Schrelmann type, in the framework of non-smooth critical point theory. Applications are given to the problem of the motion of a charged particle under the action of a Li\'enard-Wiechert potential and to the relativistic forced Kepler problem.
\end{abstract}
\medskip



\section{Introduction}

According to the principles of electrodynamics \cite{Ja75}, the motion of a slowly accelerated charged particle 
under the influence of an electromagnetic field is ruled by the Lorentz force equation
\begin{equation}\label{eq:lorentz_intro}
\frac{d}{dt}\left(\frac{m\dot{x}}{\sqrt{1-|\dot{x}|^{2}/c^{2}}}\right) =
q \left(E(t,x) + \dot x \times B(t,x)\right), \qquad x \in \mathbb{R}^3,
\end{equation}
where $m$ is the mass of the particle, $q$ is its charge and $c$ is the speed of light; moreover, the electric and magnetic fields $E$ and $B$ are provided by the potentials $V$ and $A$ via the usual relations
\begin{equation} \label{eq-relelettricomagnetico}
E(t,x) = - \nabla_x V(t,x) - \partial_t A(t,x), \qquad B(t,x) = \textnormal{curl}_x A(t,x).
\end{equation}
As well known (see, for instance, \cite{FLS64}) equation \eqref{eq:lorentz_intro} is formally the Euler-Lagrange equation of the action functional
$$
\int_{0}^{T} mc^2 \left( 1 - \sqrt{ 1 - \frac{\vert \dot x(t) \vert^2}{c^2}}\right) \, dt+
\int_{0}^{T} q \left( - V(t,x(t))+ A(t,x(t)) \cdot \dot x(t) \right) \, dt.
$$
In spite of this, and probably due to the lack of smoothness of the kinetic part of the above functional, a systematic investigation 
of equation \eqref{eq:lorentz_intro} with the tools of critical point theory has been initiated only very recently. More precisely, in \cite{ArBeTo19,ArBeTo20} a rigorous variational formulation in the space $W^{1,\infty}$ is introduced, allowing for the use of non-smooth critical point theory in the version developed by Skulzkin \cite{Sz86}, and, as a consequence, several existence and multiplicity results are given for solutions of equation \eqref{eq:lorentz_intro} with either Dirichlet or periodic boundary conditions (see also \cite{GaTo20} for the use of topological techniques). However, in both the papers the physically relevant case of singular electric and magnetic fields is not taken into account. 

The aim of the present paper is to provide a contribution in this direction. To this end, we take advantage of a recent research \cite{BoDaPapp} dealing with the equation 
\begin{equation} \label{eq:piano}
\frac{d}{dt}\left(\frac{m\dot{x}}{\sqrt{1-|\dot{x}|^{2}/c^{2}}}\right) = - \nabla_x V(t,x), \qquad x\in \mathbb{R}^2,
\end{equation}
which is a version of \eqref{eq:lorentz_intro} in the plane with $A \equiv 0$ and $q =1$. 
More precisely, in \cite{BoDaPapp} equation \eqref{eq:piano} with a singular potential $V$ given by
$V(t,x)=-\alpha/|x|-U(t,x)$ (with $\alpha > 0$), is considered, namely 
\begin{equation} \label{eq:piano2}
\frac{d}{dt}\left(\frac{m\dot{x}}{\sqrt{1-|\dot{x}|^{2}/c^{2}}}\right) = - \alpha \frac{x}{\vert x \vert^3} + \nabla_x U(t,x), \qquad x\in \mathbb{R}^2.
\end{equation}
Let us point out that the motivation given in \cite{BoDaPapp} for the above equation was not coming from electrodynamics, but rather from relativistic celestial mechanics: indeed, equation \eqref{eq:piano2} is interpreted as a simple model, in special relativity, for the motion of a particle in a forced Kepler potential (see, for instance, \cite{AnBa71} as well as the references in \cite{BoDaFe22}). Of course, however, this can be of interest also in the context of electromagnetism and, actually, this interpretation is even more natural, since a rigorous treatment of the theory of gravitation should require the framework of general relativity: we refer to \cite[Problem 34.3]{Gr04} for an interesting discussion and comparison about Kepler and Coulomb problems from the relativistic point of view. 

By using minimization and min-max arguments in the framework of non-smooth critical point theory, it is proved in \cite{BoDaPapp} that, for any external perturbation $U$, non singular and $T$-periodic in time, equation \eqref{eq:piano2}
has infinitely many $T$-periodic solutions and, in particular, at least two $T$-periodic solutions of winding number $k$ around the origin, for any integer $k \neq 0$. Of course, such a result deeply relies on the presence of the singularity $ x= 0$ for the potential $V$, which produces a non-trivial topology for the domain of the action functional: the set of $T$-periodic paths winding $k$ times around the origin is nothing but a connected component of the domain, and each of them (but the one with $k = 0$) carries at least two periodic solutions of \eqref{eq:piano2}.
Let us emphasize the universal character of this result, meaning that no assumptions on $U$ (besides its smoothness) are needed: this is ultimately a  consequence of the fact that a periodic path $x$ winding around the origin with bounded velocity (since $|\dot x| < c)$ is a priori-bounded.

In this paper, we provide a sort of generalization of the result in \cite{BoDaPapp} applying to the Lorentz force equation \eqref{eq:lorentz_intro}. More precisely, we consider an electrostatic potential $V<0$ defined in a set $\Omega$ of the form
\begin{equation}\label{defomegaintro}
\Omega = \{ (t,x) \in \mathbb{R} \times \mathbb{R}^3 \, : \, x \neq r_j(t), \, \forall \ j=1,\ldots,N\},
\end{equation}
where the functions $r_1,\ldots,r_N: \mathbb{R} \to \mathbb{R}^3$ are $T$-periodic (for some $T > 0$), of class $C^1$ with 
$\Vert \dot r_i \Vert_\infty < c$ and such that 
$r_i(t) \neq r_j(t)$ for every $t \in [0,T]$ and $i \neq j$.
Moreover, we assume that $V$ has a Keplerian blow-up at the boundary of $\Omega$ (cf. assumption (V) in Section \ref{sec3}) and that the magnetic potential $A$ satisfies the global condition
\begin{equation} \label{eq:intro1}
|A(t,x)|\leq -\dfrac{\kappa'}{c}\, V(t,x),\quad \forall \ (t,x)\in \Omega,
\end{equation}
for some $\kappa'\in (0,1)$. Under these conditions, if both $A$, $V$ and their derivatives tend to zero at infinity, we prove that \eqref{eq:lorentz_intro} has infinitely many $T$-periodic solutions (cf. Theorem \ref{teo:main}). 

Let us point out that the structure of the singularities of $V$, described via the set $\Omega$, is modeled on the relevant case of Li\'enard-Wiechert potentials (cf. \cite{Ja75} and Section \ref{sec4}), corresponding to the motion of a charged particle under the effect of $N$ moving charged particles $q_1,\ldots , q_N$. In this situation, the functions $r_1,\ldots r_N$ are the motions laws of the particles generating the potentials and $V$ and $A$ are given by 
\[
V(t,x)=\sum_{i=1}^N \dfrac{q_i}{4\pi \varepsilon_0}\, \dfrac{1}{1-\eta_i(t_{i},x)\cdot \beta_i(t_i)}\, \dfrac{1}{\vert x-r_i(t_{i}) \vert},\qquad \beta_i(t)=\frac{\dot{r}_i(t)}{c},\ \eta_i(t,x)=\frac{x-r_i(t)}{|x-r_i(t)|},
\]
and
\[
A_{i}(t,x)=\sum_{i=1}^N \dfrac{\beta_i(t_i)}{c}\, V_i(t,x),
\]
where $t_i = t_i(t,x)$ is the so-called retarded time (see \eqref{eq:temporitardato} in Section \ref{sec4}). In particular, let us notice that condition \eqref{eq:intro1} is satisfied since $\Vert \dot r_i \Vert_\infty < c$, for every $i=1,\ldots, N$.

As a second application of our main result, going back to the relativistic celestial mechanics framework, we can prove the existence of infinitely many $T$-periodic solutions for the relativistic forced Kepler problem in the space
\[
\frac{d}{dt}\left(\frac{m\dot{x}}{\sqrt{1-|\dot{x}|^{2}/c^{2}}}\right) =
- \alpha \frac{x}{\vert x \vert^3} + \nabla_x U(t,x), \qquad x \in \mathbb{R}^3,
\]
when $U>0$ and $U \to 0$ for $\vert x \vert \to +\infty$ together with its gradient (cf. Theorem \ref{teo:MC} in Section \ref{sec4}). In particular, this provides a partial generalization of the result given in \cite{BoDaPapp} for the planar case mentioned above.

For the proof of Theorem \ref{teo:main} we use a variational approach, combining arguments from both \cite{ArBeTo19,ArBeTo20} and \cite{BoDaPapp}. More precisely, we consider the functional $I:W^{1,\infty}_T \to (-\infty,+\infty]$ defined as
\[
I(x)=\int_{0}^{T} c^2 \left( 1 - \sqrt{ 1 - \frac{\vert \dot x(t) \vert^2}{c^2}}\right) \, dt+
\int_{0}^{T} \left( - V(t,x(t))+ A(t,x(t)) \cdot \dot x(t) \right) \, dt,
\] 
whenever $x$ belongs to the subset $\Lambda \subset W^{1,\infty}_T$ of paths without collisions (that is, $(t,x(t)) \in \Omega$ for every 
$t \in [0,T]$, where $\Omega$ is as in \eqref{defomegaintro}) and $\Vert \dot x \Vert_\infty \leq c$, and extended to $+\infty$ otherwise.
This functional satisfies the structural assumption of Skulzin non-smooth critical point theory and its critical points give rise to $T$-periodic solutions of equation \eqref{eq:lorentz_intro}; moreover, it is well-behaved near collisions, in the sense that if $x_n$ 
approaches the boundary of $\Lambda$, then $I(x_n) \to +\infty$. These properties can be proved by using arguments already developed in \cite{ArBeTo19,ArBeTo20,BoDaPapp} and are collected in Lemma \ref{pro:assumptionI}. 

On the other hand, however, due to the three-dimensional setting, the approach of \cite{BoDaPapp} based on the winding number cannot be used and a different strategy to achieve both existence and multiplicity has to be developed. 
In particular, inspired by classical results available in the setting of classical mechanics \cite{ACZbook}, we detect periodic solutions via a min-max principle of Lusternik-Schnirelmann type. For this, two main issues have to be faced. On one hand, we prove that the functional $I$ satisfies a weak form of the Palais-Smale condition at any level $c > \inf I = 0$, cf. Lemma \ref{pro:ps}. On the other hand, we show that the proper domain of the action functional $I$ contains compact subsets of arbitrarily large category, allowing us to define the min-max levels 
\[
c_j = \inf_{A \in \mathcal{F}_j} \sup_{x \in A} I(x),\qquad j\in \mathbb{N},
\]
where $\mathcal{F}_j$ is the family of compact subsets of the domain of $I$ having category at least $j$, cf. Lemma \ref{pro:compatti}; moreover, $c_j > 0$ for any $j \geq 3$, cf. Lemma \ref{pro:livelli}. Then, taking advantage of the general min-max principle for non-smooth functionals proved in \cite{BoDaPapp} (cf. Theorem \ref{thm:vecchio} in Section \ref{sec2}), we can prove that for $j \geq 3$ the number $c_j$ is a critical level for the action functional. This would ensure the existence of infinitely many periodic solutions to equation \eqref{eq:lorentz_intro} provided a sequence of distinct critical levels $c_j$ exists, a fact which however seems hard to be established in general. Thus, adapting the arguments in the proof of \cite[Th. 1]{ArBeTo20} we prove that whenever two critical levels coincide, the corresponding critical level carries infinitely many critical points. From this, we deduce that the functional $I$ has infinitely many critical points.

To the best of our knowledge, a technique of this type seems to be completely new in a non-smooth setting and we think that the general Lusternik-Schnirelmann min-max principle we introduce can be of independent interested.

The plan of the paper is the following. In Section \ref{sec2}, we describe the abstract variational setting and we provide the non-smooth min-max principle of Lusternik-Schnirelmann type (Theorem \ref{thm:minmax}). 
In Section \ref{sec3} we state and prove our main result (Theorem \ref{teo:main}).
Finally, in Section \ref{sec4} we provide the above mentioned applications: the motion of a charged particle under the influence of periodic Li\'enard-Wiechert potentials (Theorem \ref{teo:LW}) and the perturbed relativistic Kepler problem (Theorem \ref{teo:MC}).

\section{An abstract result}\label{sec2}

In this section, we present a result on the existence of infinitely many critical points for non-smooth functionals with singularities. More precisely, as in \cite{BoDaPapp} we are concerned with functionals of the form described in the following assumption.

\begin{Assumption}\label{ass:I}
	$ I:X\to (-\infty,+\infty] $ is a functional which can be decomposed as
	\[
	I(x)=\psi(x)+\Phi(x), \quad \forall \ x\in X,
	\]
	where, denoting by $D_{\psi}=\{x\in X: \psi(x)<+\infty\}$ and $D_{\Phi}=\{x\in X: \Phi(x)<+\infty\}$,
	\begin{enumerate}
		\item
		$D_{\Phi}$ is open in $X$ and $D_{I}=D_{\psi}\cap D_{\Phi}\ne\emptyset;$
		\item
		$\psi:X\to\RR\cup\{+\infty\} $ is convex and lower semi-continuous;
		moreover, $\psi$ is continuous on any nonempty compact set $A\subset X$ such that $ \sup_{A} \psi $ is finite;
		\item
		$\Phi:X\to\RR\cup\{+\infty\} $ is locally Lipschitz continuous in $D_{\Phi}$, i.e. every $x\in D_{\Phi}$ has a neighborhood in which $\Phi$ is Lipschitz continuous;
		\item for any sequence $\{x_n\}$ in $D_{I}$ such that $\dist(x_n, \partial D_{\Phi}) \to 0$, it holds that
		$I(x_n) \to +\infty$.
	\end{enumerate}
\end{Assumption}
\noindent
We now recall some basic definitions from \cite[\S 3.2]{MoPa99}. 
\begin{Definition}\label{def:puntocritico}
	Let $ I : X\to (-\infty,+\infty] $ satisfy Assumption~\ref{ass:I}.
	\begin{enumerate}
		\item
		A point $x \in D_{I}$ is a \emph{critical point} of $I$ if
		\[
		\Phi^{0}(x;z-x) + \psi(z) - \psi(x) \ge 0, \quad \forall \, z \in X,
		\]
where
\[
\Phi^{0}(x;u)\coloneqq\limsup_{w\to x, t\to 0^{+}}\frac{\Phi(w+tu)-\Phi(w)}{t}.
\]
		\item
		A \emph{Palais-Smale} (abbreviated \emph{PS-}) \emph{sequence for $I$ at level $c$}
		is a sequence $\{x_{n}\}$ in $X$ such that $I(x_{n})\to c $ and
		\begin{equation}\label{PS-formula}
		\Phi^{0}(x_{n};z-x_{n})+\psi(z)-\psi(x_{n}) \ge-\epsilon_{n}\|z-x_{n}\|, \quad \forall \, n\in\NN,\, z\in X,
		\end{equation}
		for some sequence $\epsilon_{n}\to 0^{+}$.
	\end{enumerate}
\end{Definition}
\begin{Remark}\label{rem:loclip}
The functional $\Phi$ in Section~\ref{sec3} is actually of class $C^{1}$ in its domain $D_{\Phi}$
and, thus,
\[
\Phi^{0}(x;u)=d\Phi(x)[u].
\]
We decided to present this more abstract section in the setting of nonsmooth calculus since, on one hand, the assumption $\Phi\in C^{1}(D_{\Phi})$ doesn't really simplify the argument and, on the other,
locally Lipschitz functionals immediately appear as soon as one considers some truncation of a $C^{1}$ functional.
\end{Remark}
We also need to consider the following weak form of the Palais-Smale condition, as introduced in \cite{ArBeTo19}.

\begin{Definition}\label{def:weakPS}
	Let $ I : X\to (-\infty,+\infty] $ satisfy Assumption~\ref{ass:I} and assume that there exists a Banach space $Y$ such 
	that $X \subset Y$ with continuous embedding. The functional $I$ is said to satisfy the \emph{weak Palais-Smale condition at level $c$} if for every PS-sequence $\{x_{n}\}$ in $X$ such that $I(x_{n})\to c $, there exist $x \in X$ and a subsequence $\{x_{n_k}\}$ such that 
	$x$ is a critical point of $I$ with $I(x) = c$ and $x_{n_k} \to x$ in the $Y$-topology.
\end{Definition}

The existence of infinitely many critical points for a functional of the form $I$ is obtained 
using a general non-smooth min-max principle, together with the Lusternik-Schnirelmann category.
For the readers convenience, we recall here the definition and basic properties of the category (cf. \cite{ACZbook}) and the min-max principle we use (established in \cite{BoDaPapp} as a generalization of \cite[Theorem~3.1]{LiMa04}).

Given $M\subset X$, the category of $A\subset M$ relative to $M$, denoted by $\textrm{cat}_X (A,M)$ is the least integer $k$, if it exists, such that 
\[
A\subset A_1\cup \ldots \cup A_k,
\]
where $A_i\subset M$ is closed and contractible in $M$, for every $i=1,\ldots, k$. The category is infinite if such a least integer does not exist.
The category satisfies the following properties, which will be used in many situations:

\begin{itemize}
\item[(P1)] if $A\subset B\subset M$, then $\cat_X (A,M)\leq \cat_X (B,M)$;
\item[(P2)] if $A\subset M\subset N$, then $\cat_X (A,N)\leq \cat_X (A,M)$;
\item[(P3)] if $A,B\subset M$, then $ \cat_{X}(A\cup B,M)\le\cat_{X}(A,M)+\cat_{X}(B,M) $;
\item[(P4)] if $A\subset M$ is closed and $\varphi \in C(A,M)$ is a deformation (i.e. it is homotopic to the inclusion $\iota_A:A\to M$), then $\cat_X (A,M)\leq \cat_X (\varphi(A),M)$.
\end{itemize}

\begin{Theorem}\cite[Th. 2.4]{BoDaPapp} \label{thm:vecchio}
	Let $I=\psi+\Phi$ be a functional satisfying Assumption~\ref{ass:I}, let $B$ be a closed set in $X$ and $\mathcal{F}$ be a family of compact sets in $X$ such that:
	\begin{itemize}
		\item[(1)] 
		$\mathcal{F}$ is \emph{homotopy stable with extended boundary} $B$, that is, for each
		$ A \in \mathcal{F} $ and each continuous deformation $ \eta\in C^{0}([0,1]\times X, X) $ such that
		$$ \eta(t,x)=x, \quad \forall \ (t,x)\in (\{0\}\times X)\cup ([0,1]\times B)  \qquad \text{ and } \qquad
		\eta([0,1]\times A)\subset D_{\Phi},$$ 
		one has that
		$ \eta(\{1\}\times A)\in\mathcal{F} $;
		\item[(2)]
		$ c \vcentcolon= \adjustlimits\inf_{A\in\mathcal{F}}\sup_{x\in A} I(x) < +\infty $;
		\item[(3)]
		there exists a closed set $F$ in $X$ such that
		\[
		(A\cap F)\setminus B \ne \emptyset, \quad \forall \ A\in\mathcal{F} \qquad\text{and}\qquad
		\sup_{B} I \le \inf_{F} I.
		\]
	\end{itemize}
	Then, for any sequence $\{A_{n}\}$ in $\mathcal{F}$ such that
	$\adjustlimits \lim_{n\to\infty}\sup_{A_{n}} I = c$, there exists a PS-sequence $\{x_{n}\}\subset X$
	at level $c$ such that $ \dist(x_{n},A_{n})\to 0 $.
	If moreover $ \inf_F I = c $, then also $\dist(x_{n},F)\to 0$.
\end{Theorem}

\noindent
We are now in a position to state our result. For every integer $j\in \mathbb{N}$, let us define
\begin{equation} \label{eq:defcompatti}
\mathcal{F}_j=\{A \subset D_\Phi \, : \, A \mbox{ compact }, \, \textrm{cat}_X (A,D_\Phi) \geq j\}.
\end{equation}
Moreover, let 
\begin{equation} \label{eq:livelli}
c_j = \inf_{A \in \mathcal{F}_j} \sup_{x \in A} I(x),
\end{equation}
for every $j\in \mathbb{N}$ such that $\mathcal{F}_j$ is not-empty. Then, we are able to prove the following result.

\begin{Theorem}\label{thm:minmax}
	Let $I=\psi+\Phi$ be a functional satisfying Assumption~\ref{ass:I}
	and the weak Palais-Smale condition at each level $c > \inf I$.
	Moreover, let us assume that there exists $j_0\in \mathbb{N}$ such that
\begin{itemize}
	\item[(i)] $\mathcal{F}_j\neq \emptyset$, for every $j\geq j_0$
	\item[(ii)] $c_j < +\infty$ for every $j \geq j_0$
	\item[(iii)] $c_{j_0} >\inf I$.
\end{itemize}
Then, the functional $I$ has infinitely many critical points. More precisely:
\begin{itemize}
	\item[(a1)] $c_j$ is a critical level of $I$, for every $j\geq j_0$
	\item[(a2)] whenever $c_{j_1} = c_{j_2}$ for some $j_2 > j_1 \geq j_0$, then the functional $I$ has infinitely many critical points at level $c_{j_1}$.
\end{itemize}
\end{Theorem}

\begin{proof} \textit{(a1)} Let us fix $j\geq j_0$. We claim that the assumptions of Theorem \ref{thm:vecchio} with $B=\emptyset$, $F=X$ and $\mathcal{F}=\mathcal{F}_j$ are satisfied. Indeed, assumption (1) is a consequence of property (P4) of the category. Moreover, assumption (2) is guaranteed by (ii) and assumption (3) is trivially fulfilled since 
	\[
	(A\cap F)\setminus B =A \ne \emptyset, \quad \forall \ A\in\mathcal{F}_j
	\]
	and
	\[
	B=\emptyset \quad \Longrightarrow \quad \sup_{B} I=-\infty.
	\]
Hence, we can apply Theorem \ref{thm:vecchio} to obtain the existence of a PS-sequence at the level $c_j\geq c_{j_0}>\inf I$. Since $I$ satisfies the weak Palais-Smale condition at levels greater than $\inf I$, we deduce that there exists a critical point at level $c_j$.
\medskip

\noindent
\textit{(a2)} The argument here follows closely the one in \cite[Theorem~1]{ArBeTo20}, with three main changes.
The first one is that our functional $\Phi$ is singular and locally Lipschitz continuous, instead of being even and in  $C^{1}(X)$. 
The second is that we use the Lusternik-Schnirelmann category instead of the Krasnoselskii genus, a fact which is, however, linked to the first difference.
Finally, we use item 2 in Assumption \ref{ass:I}, which is weaker than the continuity of $\psi$ on its proper domain $D_\psi$ required in \cite{ArBeTo20}. 

Concerning the weak PS-condition, we denote $\|\cdot\|_{Y}$ the norm in $Y$ (see Definition~\ref{def:weakPS}) and we set $B_{Y}(x,r)=\{u\in X: \|u-x\|_{Y}<r\}$ which is open in $X$ also w.r.t. the stronger topology induced by $\|\cdot\|$.

By contradiction, let us assume that $I$ has only $n\in\NN$ critical points at level $c\coloneqq c_{j_{1}}=c_{j_{2}}$, which we label $x_{1},\dots,x_{n}$, and let $r>0$ be such that the sets
$\overline{B_{Y}(x_{m},2r)}$ are pairwise disjoint and contained in $D_{\Phi}$ (the closure is taken w.r.t. the norm $\|\cdot\|$ of $X$, if not otherwise specified).
We define
\[
N_{\rho}=B_{Y}(x_{1},\rho)\cup\dots\cup B_{Y}(x_{n},\rho),\quad \forall \,\rho>0,
\]
and observe that, arguing by contradiction and using the weak PS-condition, there exists $\epsilon\in(0,r^{2})$ such that, for each $x\in I^{-1}([c-\epsilon,c+\epsilon])\setminus N_{r}$, there is $\xi_{x}\ne x$ such that
\begin{equation}\label{eq:discesa}
\psi(\xi_{x})-\psi(x)+\Phi^{0}(x,\xi_{x}-x)<-\sqrt{\epsilon}\|\xi_{x}-x\|.
\end{equation}
Let $A\in\mathcal{F}_{j_{2}}$ be chosen in such a way that
\[
\sup_{A} I \le c+\epsilon.
\]
In particular, $A\subset D_{\phi}$ and $\sup_{A}I=\max_{A}I$ since $\psi$ is bounded on $A$ and, thus, continuous in $A$ by Assumption~\ref{ass:I}.
The set $B=A\setminus N_{2r}$ is compact in $X$ and $B\in\mathcal{F}_{j_{1}}$, since
\[
j_{1}<j_{2} \le \cat_{X}(A,D_{\Phi}) \le \cat_{X}(B,D_{\Phi})+
\cat_{X}\left(\bigcup_{m=1}^{n}B_{Y}(x_{m},2r),D_{\Phi}\right) = \cat_{X}(B,D_{\Phi})+1
\]
 by (P3).
As a consequence we have
\[
c \le \max_{B}I \le \max_{A}I \le c+\epsilon.
\]
We apply Ekeland variational principle (see also \cite[Lemma~1(iii)]{ArBeTo20}), to the map $\Pi:\mathcal{F}_{j_{1}}\to\left]-,\infty+\infty\right]$ such that $\Pi(A)=\sup_{A}I$, since $\mathcal{F}_{j_{1}}$ is complete w.r.t. the Hausdorff metric
\[
\dH(A,B)=\max\{\sup_{a\in A}\dist(a,B);\sup_{b\in B}\dist(b,A)\}
\]
and $\Pi$ is lower semi-continuous w.r.t. the same matric.
Then, we obtain  $C\in\mathcal{F}_{j_{1}}$ such that
\begin{equation}\label{eq:Ekeland}
\begin{aligned}
\max_{C}I   & \le  \max_{B}I \\
\dH(B,C) & \le  \sqrt{\epsilon} < r \\
\sup_{D}I   & \ge  \max_{C}I-\sqrt{\epsilon}\delta(C,D), \quad\forall\ D\in\mathcal{F}_{j_{1}},
\end{aligned}
\end{equation}
where  is the Hausdorff distance for compact sets of a metric space.
In particular,
$C\cap N_{r}=\emptyset$ and the set $S=\{x\in C:c-\epsilon\le I(x)\}$ is contained
in $ I^{-1}([c-\epsilon,c+\epsilon])\setminus N_{r} $ and is compact in $X$.

Since the mapping
\[
(x_{1},x_{2})\mapsto \phi(\xi_{x})-\phi(x_{1})+\Phi^{0}(x_{2};\xi_{x}-x_{2})+\sqrt{\epsilon}\|\xi_{x}-x_{1}\|
\]
is upper semi-continuous in $X\times D_{\Phi}$ (by \cite[Proposition~2.1.1]{Clarke90}) and negative for $x_{1}=x_{2}=x\in S$ by \eqref{eq:discesa}, for each $x\in S$ there is a positive $\delta_{x}<\|\xi_{x}-x\|$ such that
$\overline{B_{X}(x,\delta_{x})}\subset D_{\Phi}$ and
\[
\phi(\xi_{x})-\phi(u)+\Phi^{0}(x+h;\xi_{x}-u)<-\sqrt{\epsilon}\|\xi_{x}-u\|,
\quad \forall\ u\in \overline{B_{X}(x,\delta_{x})}, \, h\in \overline{B_{X}(0,\delta_{x})}.
\]
Since $S$ is compact, there exist $y_{1},\dots,y_{\ell}\in S$ such that $S\subset B_{1}\cup\dots\cup B_{\ell}$, where $B_{k}=B_{X}(y_{k},\delta_{y_{k}})$, $1\le k \le \ell$.
We observe that $\xi_{y_{k}}\not\in B_{k}$, by construction, and, thus, we can fix some positive
$\delta \le \min\{\delta_{C}/2,\delta_{y_{k}},\dist(\xi_{y_{k}},\overline{B}_{k}\cap C): 1\le k\le \ell \}$, where $\delta_{C}\coloneqq\min\{\dist(x,\partial D_{\Phi}):x\in c\} >0$ since
$C\subset D_{\Phi}$ by \eqref{eq:Ekeland}.

Let us denote by $\eta,\eta_{k}:C\to[0,1]$ ($1\le k\le\ell$) continuous functions such that
\[
\eta(x)=\begin{cases}
1 &\text{if } I(x)\ge c \\
0 &\text{if } I(x)\le c-\epsilon
\end{cases}\quad\text{and}
\quad
\eta_{k}(x)=\begin{cases}
\dfrac{\dist(x,C\setminus B_{k})}{\sum_{m=1}^{\ell}\dist(x,C\setminus B_{m})} &\text{if }x\in B_{k}\cap C \vspace{4pt}\\
0 &\text{if } x\in C\setminus B_{k}
\end{cases}
\]
so that $\sum_{k=1}^{\ell}\eta_{k}=1$ on $S$.
Let us consider the function $\beta: [0,1]\times C \to X$ defined by
\[
\begin{aligned}
\beta(t,x)=\beta_{t}(x) &\coloneqq x+t\delta\eta(x)
\sum_{k=1}^{\ell}\frac{\eta_{k}(x)}{\|\xi_{y_{k}}-x\|}(\xi_{y_{k}}-x) \\
&= \left[ 1-t\delta\eta(x)\sum_{k=1}^{\ell}\frac{\eta_{k}(x)}{\|\xi_{y_{k}}-x\|}\right]x
+t\delta\eta(x)\sum_{k=1}^{\ell}\frac{\eta_{k}(x)}{\|\xi_{y_{k}}-x\|}\xi_{y_{k}}
\end{aligned}
\]
which is continuous and satisfies $\|\beta_{t}(x)-x\|\le\delta<\delta_{C}$ for all $(t,x)\in[0,1]\times C$  by construction.
As a consequence $\beta_{1}$ is a deformation of $C$ in $D_{\Phi}$ (observe that $\beta_{0}$ is the identity on $C$) and $D\coloneqq\beta_{1}(C)$ belongs to $\mathcal{F}_{j_{1}}$ by property (P4).

From the estimate
\[
t\delta\eta(x)\sum_{k=1}^{\ell}\frac{\eta_{k}(x)}{\|\xi_{y_{k}}-x\|}
\le \delta\sum_{k=1}^{\ell}\frac{\eta_{k}(x)}{\dist(\xi_{y_{k}},\overline{B}_{k}\cap C)}
\le \sum_{k=1}^{\ell}\eta_{k}(x)\le 1,
\]
we deduce that
$\beta_{t}(x)$ is a convex combination of $x,\xi_{y_{1}},\dots,\xi_{y_{\ell}}$ and, hence,
\[
\psi(\beta_{t}(x))\le
\left[1-t\delta\eta(x)\sum_{k=1}^{\ell}\frac{\eta_{k}(x)}{\|\xi_{y_{k}}-x\|}\right]\psi(x)
+t\delta\eta(x)\sum_{k=1}^{\ell}\frac{\eta_{k}(x)}{\|\xi_{y_{k}}-x\|}\psi(\xi_{y_{k}}).
\]
On the side of $\Phi$, by Lebourg's theorem \cite[Theorem~2.3.7]{Clarke90} for each $x\in C$ there
exists $\tau=\tau(x)\in(0,1)$ and $\zeta\in\partial\Phi(\beta_{\tau}(x))$ such that $\Phi(\beta_{1}(x))-\Phi(x)=\langle \zeta,\beta_{1}(x)-x\rangle $, where $\partial\Phi(x)$
is the generalized gradient of $\Phi$ at $x$ (see \cite[\S 2.1]{Clarke90}).
Hence, we have
\[
\Phi(\beta_{1}(x))-\Phi(x)\le \Phi^{0}(\beta_{\tau}(x);\beta_{1}(x)-x) \le
\delta\eta(x)\sum_{k=1}^{\ell}\frac{\eta_{k}(x)}{\|\xi_{y_{k}}-x\|}
\Phi^{0}(\beta_{\tau}(x);\xi_{y_{k}}-x)
\]
by \cite[Propositions~2.1.1-2]{Clarke90}.
As a consequence we can estimate
\[
\begin{aligned}
I(\beta_{1}(x)) &\le  \left[1-\delta\eta(x)\sum_{k=1}^{\ell}\frac{\eta_{k}(x)}{\|\xi_{y_{k}}-x\|}\right]\psi(x)
+\delta\eta(x)\sum_{k=1}^{\ell}\frac{\eta_{k}(x)}{\|\xi_{y_{k}}-x\|}\psi(\xi_{y_{k}}) \\
&\hphantom{\quad\;} 
+\Phi(x)+\delta\eta(x)\sum_{k=1}^{\ell}\frac{\eta_{k}(x)}{\|\xi_{y_{k}}-x\|}
\Phi^{0}(\beta_{\tau}(x);\xi_{y_{k}}-x) \\
&= I(x)+\delta\eta(x)\sum_{k=1}^{\ell}\frac{\eta_{k}(x)}{\|\xi_{y_{k}}-x\|}
\left[ \psi(\xi_{y_{k}})-\psi(x)+\Phi^{0}(\beta_{\tau}(x);\xi_{y_{k}}-x)\right].
\end{aligned}
\]
Now, by construction, $\|\beta_{\tau}(x)-x\|\le\delta\le\delta_{y_{k}}$ for all $k=1,\dots,\ell$,
which implies that
\[
\psi(\xi_{y_{k}})-\psi(x)+\Phi^{0}(\beta_{\tau}(x);\xi_{y_{k}}-x)<-\sqrt{\epsilon}\|\xi_{y_{k}}-x\|,
\quad\forall\ x\in\overline{B}_{k}, \,\ k=1,\dots,\ell.
\]
Therefore, we have
\[
I(\beta_{1}(x))<I(x)-\delta\eta(x)\sqrt{\epsilon}\sum_{k=1}^{\ell}\eta_{k}(x)
< I(x)-\delta\eta(x)\sqrt{\epsilon} \quad\forall\ x\in S.
\]
On the other hand, if $x\in C\setminus S$, we have $\eta(x)=0$ and $I(\beta_{1}(x))=I(x)<c-\epsilon$.

We can choose $x_{0}\in C$ such that $I(\beta_{1}(x_{0})) = \max I(\beta_{1}(C))=\max I(D)\ge c$ since $D\in\mathcal{F}_{j_{1}}$.
As consequences we have that $x_{0}\in S$ and
\[
c\le\max_{D} I = I(\beta_{1}(x_{0})) <I(x_{0})-\delta\eta(x_{0})\sqrt{\epsilon}\le I(x_{0}),
\]
and, thus, $\eta(x_{0})=1$ and
\[
\max_{D} I < I(x_{0})-\delta\sqrt{\epsilon} \le \max_{C}I-\delta\sqrt{\epsilon}
\le \max_{C}I - \dH(C,D)\sqrt{\epsilon},
\]
which is a contradiction with \eqref{eq:Ekeland}.
\end{proof}
%
%
\section{The main result}\label{sec3}
In this section we state and prove our main result for the Lorentz force equation
\begin{equation}\label{eq:lorentz}
\frac{d}{dt}\left(\frac{\dot{x}}{\sqrt{1-|\dot{x}|^{2}/c^{2}}}\right) =
E(t,x) + \dot x \times B(t,x), \qquad x \in \mathbb{R}^3,
\end{equation}
where, as usual,
\begin{equation}\label{formulaEB}
E(t,x) = - \nabla_x V(t,x) - \partial_t A(t,x), \qquad B(t,x) = \textnormal{curl}_x A(t,x).
\end{equation}
Notice that, without loss of generality, we have normalized the charge-to-mass ratio to $1$ (while, on the other hand, we prefer to keep track of the value $c$ of the speed of light).

As already mentioned in the Introduction, our main interest is in covering the case when the potential $V$ is singular. 
More precisely, we assume that the singularities of $V$ are described by $N$ functions $r_1,\ldots,r_N: \mathbb{R} \to \mathbb{R}^3$ which are $T$-periodic (for some $T > 0$), of class $C^1$, with $\|\dot{r}_{i}\|_{\infty}<c$, and such that 
$r_i(t) \neq r_j(t)$ for every $t \in [0,T]$ and $i \neq j$. Accordingly, we settle equation \eqref{eq:lorentz} on the open domain
$$
\Omega = \{ (t,x) \in \mathbb{R} \times \mathbb{R}^3 \, : \, x \neq r_j(t), \, \forall \ j=1,\ldots,N\}.
$$
The following result holds true.
\begin{Theorem}\label{teo:main}
Let us assume that $V:  \Omega \to \mathbb{R}$ and $A:\Omega \to 
\mathbb{R}^3$ are of class $C^1$, $T$-periodic in the first variable, and satisfy the following conditions:
\begin{itemize}
	\item[(V)] $V(t,x) < 0$ for every $(t,x) \in \Omega$ and there exist $\kappa > 0$ and $\delta > 0$ such that, for every $i =1,\ldots,N$,
	\begin{equation} \label{eq:Vsing}
	V(t,x) \leq - \frac{\kappa}{\vert x - r_i(t) \vert}, \quad \forall \ (t,x) \in \Omega \mbox{ such that } \vert x - r_i(t) \vert < \delta;
	\end{equation}
	\item[(AV1)] there exists $\kappa' \in (0,1)$ such that
	$$
	\vert A(t,x) \vert \leq -\frac{\kappa'}{c} V(t,x), \quad \forall \ (t,x) \in \Omega;
	$$
	\item[(AV2)] it holds that	
	$$
	\lim_{\vert x \vert \to \infty} \left( |V(t,x)| + \vert \nabla_x V(t,x) + \partial_t A(t,x)\vert +\vert\textnormal{curl}_x A(t,x) \vert \right) = 0,
	$$
	uniformly in $t\in \RR$.
\end{itemize}
	Then, equation \eqref{eq:lorentz} has infinitely many $T$-periodic solutions.
\end{Theorem}

\begin{Remark}
Notice that the potential $A$ can be either regular or singular: however, in this last case, the behavior of $A$ near the singularities has to be consistent with assumption (AV1). Let us observe that the case $A \equiv 0$ is allowed.
\end{Remark}

\begin{Remark}
Let us point out that assumption (AV2) can be replaced by
$$
\lim_{\vert x \vert \to \infty} \left( |V(t,x)| + \vert \nabla_x V(t,x)\vert  + \vert D_x A(t,x)\vert \right) = 0,
$$
uniformly in $t\in \RR$. Indeed, the condition 
\begin{equation} \label{eq-nuovaipo}
\lim_{\vert x \vert \to \infty} \left( \vert \nabla_x V(t,x) + \partial_t A(t,x)\vert +\vert\textnormal{curl}_x A(t,x) \vert \right) = 0,
\end{equation}
uniformly in $t\in \RR$, is used to prove the validity of the weak Palais-Smale condition (cf. the proof Lemma \ref{pro:ps} and, in particular, formula \eqref{eq-formulaqui}, which in turn is obtained from \eqref{eq-formulaqui0} using the expression for $d\Phi$ given by \eqref{dphi2}). When assuming, instead of \eqref{eq-nuovaipo}, the condition
\[
\lim_{\vert x \vert \to \infty} \left( \vert \nabla_x V(t,x)\vert  + \vert D_x A(t,x)\vert  \right) = 0,
\]
uniformly in $t\in \RR$, then the same conclusion can be obtained using \eqref{dphi1} instead of \eqref{dphi2} in formula \eqref{eq-formulaqui0}. We prefer to suppose (AV2) because it can be verified in a more direct way in the application to Li\'enard-Wiechert potentials (cf. \eqref{eq:stimainfinitoelettrico} and \eqref{eq:stimainfinitomagnetico}).
\end{Remark}

The rest of the section is devoted to the proof of Theorem \ref{teo:main}, which follows from the abstract result Theorem \ref{thm:minmax}.

So, let us first describe the variational setting; in what follows, we take advantage of results given both in \cite{ArBeTo19} (where, however, $V$ and $A$ are not allowed to be singular) and in \cite{BoDaPapp} (where $A = 0$, but $V$ is singular).
Let us consider the Banach space
\[
X=\left\{ x\in W^{1,\infty}(0,T;\RR^{3}) : x(0)=x(T) \right\},
\]
endowed with its usual norm $ \|x\| = \|x\|_{\infty} + \|\dot{x}\|_{\infty}$. 
We define the functional
$\psi: X \to (-\infty,+\infty]$ as
$$
\psi(x) = \begin{dcases}
\int_{0}^{T} c^2 \left( 1 - \sqrt{ 1 - \frac{\vert \dot x(t) \vert^2}{c^2}}\right) \, dt
& \text{if } \Vert \dot x \Vert_{\infty} \leq c; \\
+\infty & \text{otherwise}.
\end{dcases}
$$
According to the notation of Section \ref{sec2}, we thus have
$$
D_\psi = \{ x \in X : \Vert \dot x \Vert_{\infty} \leq c\}.
$$
Moreover, we consider the open subset of $X$
$$
\Lambda = \{ x \in X : (t,x(t)) \in \Omega, \quad \forall \ t \in [0,T]\}
$$
and we define $ \Phi:X\to(-\infty,+\infty]$ as
$$
\Phi(x) = \begin{dcases}
\int_{0}^{T} \left( - V(t,x(t))+ A(t,x(t)) \cdot \dot x(t) \right) \, dt & \text{if } x \in \Lambda; \\
+\infty & \text{otherwise}, 
\end{dcases}
$$
so that $D_\Phi = \Lambda$.
Finally, we define the action functional $I: X \to(-\infty,+\infty]$ as
$$
I(x) = \psi(x) + \Phi(x), \quad \forall \ x \in X,
$$
and we recall the notation $D_I = D_\psi \cap D_\Phi$. 

For further convenience, we observe that
 assumption (AV1) implies that
	\begin{equation}\label{eq:relazione}
	 - V(t,x(t))+ \dot x(t) \cdot A(t,x(t)) \geq (1-\kappa') \, \left( -V(t,x(t))\right), \quad \forall \ x\in D_I, \, t \in [0,T],
	\end{equation}
	and then, by assumption (V),
\begin{equation}\label{segnophi}
\Phi(x) =\int_0^T \left( - V(t,x(t))+ \dot x(t) \cdot A(t,x(t))\right) \,dt >0, \quad \forall \ x \in D_I.
\end{equation}
Taking into account that $\psi\geq 0$ and that $I=+\infty$ outside $D_I$, we deduce that
\begin{equation}\label{segnoI}
I(x)  >0,\quad \forall \ x \in X.
\end{equation}

In the next Lemma, we show that this functional satisfies the structural Assumption~\ref{ass:I} of Section \ref{sec2} and that, moreover, its critical points correspond to classical $T$-periodic solutions of the Lorentz force equation \eqref{eq:lorentz}.

\begin{Lemma}\label{pro:assumptionI}
	The functional $ I $ satisfies Assumption~\ref{ass:I}. Moreover, the functional $\psi$ is lower semicontinuous with respect to uniform convergence, namely: if $x \in X$ and $\{x_n\}$ is a sequence in $D_\psi$ such that $x_n \to x$ uniformly on $[0,T]$, then 
	$x \in D_\psi$ and 
	$$
	\psi(x) \leq \liminf_{n \to +\infty} \psi(x_n).
	$$
	Moreover, each critical point $x \in D_I$ of $I$ satisfies $\vert \dot x(t) \vert < c$ for every $t \in [0,T]$ and correspond to a classical
	$T$-periodic solution of equation \eqref{eq:lorentz}.
\end{Lemma}

\begin{proof}
	Most of the above statement has been already proved in \cite[Proposition 3.2]{BoDaPapp} (and, in turn, in corresponding results in \cite{ArBeTo19}); notice indeed that the functional $\psi$ is the same as the one considered therein, while $\Phi$, despite the presence of the magnetic term, is still of class $C^1$ on the open set $D_\Phi = \Lambda$, with
\begin{equation}\label{dphi1}
d\Phi(x)[y] =\displaystyle \int_{0}^{T} \left( - \nabla_x V(t,x(t)) \cdot y(t) + A(t,x(t)) \cdot \dot{y}(t) + ((D_x A(t,x(t))^T \dot x(t)) \cdot y(t) \right) \, dt,
\end{equation}
cf. \cite[Lemma 1]{ArBeTo19} (in the above formula, the term $(D_x A)^T \dot x$ is meant as the product of the transpose of the Jacobian matrix $D_x A$ with the (column) vector $\dot x$).
Notice that, by integrating by parts,
$$
\int_0^T A(t,x(t)) \cdot \dot{y}(t) \,dt = -\int_0^T \partial_t A(t,x(t)) \cdot y(t) \,dt
- \int_0^T (D_x A(t,x(t)) \dot x(t)) \cdot y(t) \,dt
$$
From this, together with the identity
$$
((D_x A(t,x(t))^T  \dot x(t)) \cdot y(t) - (D_x A(t,x(t)) \dot x(t)) \cdot y(t) = (\dot x(t) \times \textnormal{curl}_x A(t,x(t))) \cdot y(t),
$$
we can rewrite $d\Phi$ in the equivalent form
\begin{equation}\label{dphi2}
d\Phi(x)[y] =
 \int_{0}^{T} \left( (- \nabla_x V(t,x(t))  - \partial_t A(t,x(t))) \cdot y(t) + (\dot x(t)\times {\rm curl}_x A(t,x(t))) \cdot y(t) \right) \, dt.
\end{equation}
Using this formula, and recalling \eqref{formulaEB}, the fact that critical points of $I$ gives rise to classical $T$-periodic solutions of equation \eqref{eq:lorentz} can be proved with the very same arguments of \cite[Proposition 3.3]{BoDaPapp} (see also \cite[Theorem 2]{ArBeTo19}).
	
	The only point which requires a bit of care is the proof of the property of blow-up on the boundary (that is, item 4 of Assumption ~\ref{ass:I}), for which we give the complete details. 
	At first, we notice that 
	$$
	\partial D_{\Phi} = X \setminus \Lambda = \{x \in X : \ \exists \, i \in \{1,\ldots , N\} \ \exists \, t_0\in [0,T]:\ x(t_0) = r_i(t_0)\}.
	$$
	So, let us consider a sequence $\{x_n\}$ in $D_I$ such that $d_n :
	= \dist(x_n, \partial D_{\Phi}) \to 0$
	and, accordingly, let $y_n \in \partial D_{\Phi}$ be such that $\Vert x_n - y_n \Vert \leq 2d_n$.
	Since $\Vert \dot x_n \Vert_\infty \leq c$ for any $n$, we find that
	$$
	\Vert \dot y_n \Vert_\infty \leq c + \Vert \dot y_n - \dot x_n  \Vert_\infty \leq c + \Vert  y_n - x_n  \Vert \leq c + d_n \leq c+1
	$$
	for $n$ large enough. Moreover, since $y_n(t_n) = r_{i_n}(t_n)$ for some $t_n \in [0,T]$ and $i_n \in \{1,\ldots, N\}$, we have
	$$
	\Vert y_n -r_{i_n}\Vert_\infty \leq (2c+1)T
	$$
	and thus the sequence $\{y_n\}$ is bounded in $X$. Since $\Vert x_n - y_n \Vert \leq 2d_n$, the sequence $\{x_n\}$ is bounded in $X$ as well. Therefore, the Ascoli-Arzelà theorem yields the existence of a continuous function $z$ such that, up to subsequence, $x_n \to z$ and $y_n \to z$ uniformly on $[0,T]$. Hence, $z(0) = z(T)$ and $z(t_0) = r_{i_0}(t_0)$ for some $t_0 \in [0,T]$, limit point of the sequence $t_n$, and $i_0 \in \{1,\ldots ,N\}$, limit point of the sequence $i_n$. Moreover, passing to the limit in the Lipschitz-continuity condition
	$$
	\vert x_n(t_2) - x_n(t_1) \vert \leq c |t_2-t_1|, \quad \text{ for every } t_1,t_2 \in [0,T],
	$$
	we easily see that $z \in D_\psi \subset X$. Hence, the function $z - r_{i_0}$ is Lipschitz continuous and so
	$$
	\int_{0}^{T} \dfrac{1}{|z(t)-r_{i_0}(t)|}\, dt = +\infty.
	$$
	Therefore, by (V) and (AV1) and using Fatou's lemma we obtain
	$$
	\begin{array}{l}
	\displaystyle \liminf_{n \to +\infty} \int_0^T \left(-V(t,x_n(t))+ A(t,x_n(t)) \cdot \dot{x}_n(t)\right)\, dt \geq \liminf_{n \to +\infty} \int_0^T \left(-(1-\kappa') \, V(t,x_n(t))\right)\, dt \\
	\\
	\displaystyle \geq (1-\kappa') \int_0^T \liminf_{n \to +\infty} \left(-V(t,x_n(t))\right)\, dt 
	\displaystyle \geq \kappa \, (1-\kappa') \int_0^T \liminf_{n \to +\infty} \dfrac{1}{|x_n(t)-r_{i_0}(t)|}\, dt \\
	\\
	\displaystyle
	= \kappa \, (1-\kappa') \int_0^T \dfrac{1}{|z(t)-r_{i_0}(t)|}\, dt= +\infty.
	\end{array}
	$$
	Since $0 \leq \psi(x_{n}) \leq mc^2 T$, we finally conclude that $I(x_n) \to +\infty$ as desired.
\end{proof}

Let us now notice that 
$$
\inf_X I = 0.
$$
Indeed, we have already observed that $I>0$, cf. \eqref{segnoI}. Moreover, for a sequence $x_n(t) \equiv \xi_n$ with $\vert \xi_n \vert \to +\infty$ we readily see, by assumption (AV2), that
$I(x_n) = \Phi(x_n) \to 0$. With this in mind, the next result ensures that the functional $I$ satisfies, at each level $c > \inf I = 0$, the weak Palais-Smale condition, according to Definition \ref{def:weakPS} with $Y = L^\infty(0,T)$.

\begin{Lemma} \label{pro:ps}
	The functional $I$ satisfies the weak Palais-Smale condition at each level $c > 0$.
\end{Lemma}

\begin{proof}
	Let $\{x_n\} \subset X$ be a Palais-Smale sequence at level $c > 0$; incidentally, let us notice that $\{x_n\} \subset D_I$, since otherwise $I(x_n) = +\infty$. The proof will be divided in two steps.
	
	At first, we show that the sequence $\{x_n\}$ is bounded in $L^\infty$ (and, thus, in $X$). To see this, let us write 
	$x_n = \tilde x_n + \bar{x}_n$, where $\bar{x}_n = \tfrac{1}{T}\int_0^T x_n$ and $\int_0^T \tilde x_n \,dt = 0$.
	Since $\Vert \dot{\tilde x}_n \Vert_\infty = \Vert \dot x_n \Vert_\infty \leq c$, we have that $\Vert \tilde x_n \Vert_\infty$ is bounded. 
	So, assuming by contradiction that $\Vert x_n \Vert_\infty$ is not bounded yields, up to subsequences, $\vert \bar{x}_n \vert \to +\infty$.
	Then $\vert x_n(t) \vert \geq \vert \bar{x}_n \vert - \Vert \dot{\tilde} x_n \Vert_\infty$
	and so 
	\begin{equation}\label{PS-min}
	\min_t \vert x_n(t) \vert \to +\infty.
	\end{equation}
	Choosing $z = \bar{x}_n$ in \eqref{PS-formula}, we obtain
\begin{equation} \label{eq-formulaqui0}
	d\Phi(x_{n})[-\tilde x_n]+\psi(\bar{x}_n)-\psi(x_{n}) \ge-\epsilon_{n}\| \tilde x_n \|, 
	\quad \forall \ n\in\NN,
\end{equation}
that is, using \eqref{dphi2}, 
\begin{equation} \label{eq-formulaqui}
	\psi(x_n) \leq  	
\epsilon_{n}\| \tilde x_n \| + \int_{0}^{T} \left(  \left(\nabla_x V(t,x(t))  + \partial_t A(t,x_n(t))\right) \cdot \tilde x_n(t) + (\dot x_n(t)\times {\rm curl}_x A(t,x_n(t))) \cdot \tilde x_n(t) \right) \, dt.
	\end{equation}
Therefore, recalling the boundedness of $\Vert \tilde x_n \Vert_\infty$, \eqref{PS-min} and assumption (AV2) we obtain $\psi(x_n) \to 0$.
	On the other hand, for the same reasons $\Phi(x_n) \to 0$ and so
	$$
	\psi(x_n) = I(x_n) - \Phi(x_n) \to c > 0, 
	$$
	a contradiction.
	
	As a second step, we show that the boundedness of $\{x_n\}$ implies the existence of a subsequence $\{x_{n_k}\}$ converging
	in $L^\infty(0,T)$ to a critical point $x$ of the functional $I$ at level $c$ (that is, the condition required in the definition of weak Palais-Smale condition at level $c$). For this, we combine the arguments used in the proof of \cite[Proposition 3.5]{BoDaPapp} with the ones in the proof of \cite[Lemma 5]{ArBeTo19}.
\end{proof}

We now consider the sets $\mathcal{F}_j$ and the min-max levels $c_j$ defined respectively in \eqref{eq:defcompatti} 
and \eqref{eq:livelli} and we turn to the proof of the validity of assumptions (i)-(ii)-(iii) of Theorem \ref{thm:minmax}.

At first, we deal with (i)-(ii). 

\begin{Lemma} \label{pro:compatti}
	For every integer $j \geq 1$, it holds that:
\begin{itemize}
\item[(i)] $\mathcal{F}_j\neq \emptyset$,
\item[(ii)] $c_j < +\infty$.
\end{itemize}
\end{Lemma}

\begin{proof}
We first prove that (i) holds when there is only one curve $r_{1}$ of singularities.
To this aim, we make use of the following auxiliary open sets in $C_{T}\coloneqq\{x\in C([0,T],\RR^{3}):x(0)=x(T)\}$, endowed with the topology of uniform convergence:
\[
\begin{aligned}
\Lambda_1 &= \{x\in C_{T} : x(t)\neq r_1(t),\quad \forall \ t\in [0,T]\} \\
\Lambda_0 &= \{x\in C_{T} : x(t)\neq 0,\quad \forall \ t\in [0,T]\}
\end{aligned}
\]
and of the the continuous and dense immersion $\iota:X\to C_{T}$.
We have that $X\cap\Lambda_{1} = \iota^{-1}(\Lambda_{1})$ and that $\iota|_{X\cap\Lambda_{1}}:X\cap\Lambda_{1}\to\Lambda_{1}$ is a homotopy equivalence by \cite[Theorem~16]{Palais66}.
Since the affine isometry $ x\mapsto x-r_{1} $ maps $\Lambda_1$ onto $\Lambda_0$, we have that its composition with $\iota|_{X\cap\Lambda_{1}}$ provides a homotopy equivalence between $X\cap\Lambda_{1}$ and $\Lambda_{0}$.
Using \cite[Corollary~2.8]{FaHu92}, we deduce that the cup length in $\ZZ_{2}$ of $\Lambda_{0}$ is infinite.
Since the cup length is a homotopy invariant, we infer that $X\cap\Lambda_{1}$ contains compact sets with arbitrarily large category by \cite[Lemma~2.9]{FaHu89}.

Now, we are going to show that, for each $ j\ge 1 $, there exists a compact $A\subset\Lambda\cap D_{\psi}$ such that $\cat_{X}(A,\Lambda)\ge j$; this will imply both (i) and (ii), since our functional $\psi$ is bounded in $D_{\psi}$.
We just showed that there exists a compact $A_{1}\subset\Lambda_{1}$ such that
$\cat_{X}(A_{1},\Lambda_{1}\cap X)\geq j$.
For each $\lambda >0$ and $x\in X$ we define $x_{\lambda}=r_{1}+\lambda(x-r_{1})$ and observe
that $x_{\lambda}\in\Lambda_{1}\cap X$ if and only if $x\in\Lambda_{1}\cap X$.
We set $A_{\lambda}=\{x_{\lambda}:x\in A_{1}\}$ which is compact and omeomorphic to $A_{1}$ so that
\[
\cat_{X}(A_{\lambda},\Lambda_{1}\cap X)=\cat_{X}(A_{1},\Lambda_{1}\cap X)\ge j
\quad\forall\ \lambda>0.
\]
Now, let $\delta\coloneqq\min\{|r_{1}(t)-r_{j}(t)|:t\in[0,T],\,j=2,\dots,N\}>0$.
Since
\[
\|x_{\lambda}-r_{1}\|=\lambda\|x-r_{1}\|\le\lambda\dist(r_{1},A_{1})<+\infty
\quad\forall\ x\in A_{1} \text{ and }\forall\ \lambda>0,
\]
for $\lambda<\delta/\dist(r_{1},A_{1})$ we have that $A_{\lambda}\subset \Lambda$ and
\[
\cat_{X}(A_{\lambda},\Lambda) \ge \cat_{X}(A_{1},\Lambda_{1}\cap X)\ge j
\]
by property (P2).
On the other hand, we have that
\[
\|\dot{x}_{\lambda}\|_{\infty}\le \|\dot{r}_{1}\|_{\infty}+\lambda\|\dot{x}-\dot{r}_{1}\|_{\infty}
\le c \quad \forall\ x\in A_{1}
\quad\text{if }\lambda\le \frac{c-\|\dot{r}_{1}\|_{\infty}}{\max_{x\in A_{1}}\|\dot{x}-\dot{r}_{1}\|_{\infty}}.
\]
Hence, if $\lambda>0$ is small enough we have that $A_{\lambda}\subset D_{\psi}\cap\Lambda$ and
$A_{\lambda}\in \mathcal{F}_{j}$.
\end{proof}

Finally, we prove that (iii) of Theorem \ref{thm:minmax} is satisfied with $j_0 = 3$ (while it can be shown that $c_1 = c_2 = 0$).

\begin{Lemma} \label{pro:livelli} 
It holds that $c_3>0$.
\end{Lemma}

\begin{proof}Suppose by contradiction that $c_3=0$, that is
$$
\inf_{A \in \mathcal{F}_3} \sup_{x \in A} I(x)=0.
$$
Then, for every $n\in \NN$ there exists $A_n\in \mathcal{F}_3$ such that
\[
0\leq \sup_{x \in A_n} I(x) <\dfrac{1}{n}.
\]
Of course, $A_n\subset D_I$. Hence, taking into account that $I=\psi+\Phi$, with $\psi \geq 0$ and $\Phi > 0$ in $D_I$ (cf. \eqref{segnophi}), we get
\begin{equation} \label{eq:pr102}
0\leq \psi(x) <\dfrac{1}{n}\quad \mbox{ and } \quad 0 \leq \Phi(x) <\dfrac{1}{n},\quad \forall \ x\in A_n.
\end{equation}
In particular, noticing that
\[
\psi (x)\geq \dfrac{1}{2}\, \int_0^T |\dot{x}(t)|^2\, dt,\quad \forall \ x\in A_n,
\]
we obtain
\begin{equation} \label{eq:pr103}
||\dot{x}||_{L^2} <\sqrt{\dfrac{2}{n}},\quad \forall \ x\in A_n,
\end{equation}
and so, from Sobolev inequality (see, for instance, \cite[Proposition 1.3]{MaWi89}),
\begin{equation}\label{sobolev}
\vert \tilde x(t) \vert \leq \sqrt{\frac{T}{6n}}, \quad \forall \ x\in A_n, \ t \in [0,T],
\end{equation}
where we have written as usual $x(t) = \bar{x} + \tilde x(t)$, with $\bar{x} = \tfrac{1}{T} \int_0^T x(t)\,dt$.

Now, we claim that, fixed an arbitrary constant $R > 0$ with the property that
\begin{equation} \label{eq:normalizzazione}
\max \{|r_i(t)|:\ t\in [0,T],\ i=1,\ldots, N\}\leq \dfrac{R}{2},
\end{equation}
there exists $n^*\in \NN$ such that, for every $n \geq n^*$,
\begin{equation} \label{eq:pr105}
x\in A_n \quad \Longrightarrow \quad |x(t)| \geq 2R, \quad \forall \ t \in [0,T].
\end{equation}
Indeed, let $\iota^*$ be defined by
\[
\iota^*=\inf \{-V(t,x):\ t\in [0,T], \ |x|<3R,\ (t,x)\in \Omega \}
\]
and observe that $\iota^*>0$ by assumption (V); moreover, let $n^*\in \mathbb{N}$ be such that
\begin{equation} \label{eq:pr103ter}
n^*\geq \max \left(\dfrac{1}{(1-\kappa')\, \iota^*\, T}, \frac{2T}{R^2}\right).
\end{equation}
Assume now by contradiction that there exist $n \geq n^*$, $x \in A_n$ and $t_0 \in [0,T]$ such that $|x(t_0)|< 2R$. Then,
from \eqref{eq:pr103} we infer that 
\[
|x(t)|\leq |x(t_0)|+\int_0^T |\dot{x}(t)|\, dt \leq |x(t_0)|+\sqrt{\frac{2T}{n}},\quad \forall \ t\in [0,T],
\]
and hence from \eqref{eq:pr103ter} we deduce that $|x(t)|<3R$, for every $t\in [0,T]$. Therefore, recalling \eqref{eq:relazione} and \eqref{eq:pr102}, we have
\[
\dfrac{1}{n^*}> \Phi(x)\geq (1-\kappa') \, \int_0^T -V(t,x(t))\,dt > (1-\kappa')\, \iota^*\, T,
\]
which contradicts \eqref{eq:pr103ter}.

At this point, we notice that from \eqref{sobolev} and \eqref{eq:pr103ter} it follows that
$$
\vert \tilde x(t) \vert \leq R, \quad \forall \ x\in A_n, \ t \in [0,T].
$$
Hence, taking into account \eqref{eq:pr105} we deduce that, for every $n \geq n^*$,
\begin{equation} \label{eq:pr105bis}
x\in A_n \quad \Longrightarrow \quad |\bar x + (1-\lambda) \tilde x(t)| \geq R, \quad \forall \ t \in [0,T], \ \forall \lambda \in [0,1].
\end{equation}
In particular, recalling \eqref{eq:normalizzazione}, $\bar{x} + (1-\lambda) \tilde x \in \Lambda$ for every $x \in A_n$ and $\lambda \in [0,1]$.
Hence, the map $H: [0,1] \times A_n \to \Lambda$ given by
$$
H(\lambda,x) = \bar{x} + (1-\lambda) \tilde x
$$
provides a deformation in $\Lambda$ of $A_n$ into $A_n' = H(1,A_n)$.
Hence, by property (P4) of the category,
\begin{equation}\label{categoria1}
\textrm{cat}_X (A_n,\Lambda)\leq \textrm{cat}_X (A_n',\Lambda).
\end{equation}
On the other hand, we observe that, setting 
$$
\Xi = \{ x \in \Lambda \, : \, x(t) \equiv c \mbox{ with } \vert c \vert \geq R\},
$$ 
from \eqref{eq:pr105bis} we have that $A_n' \subset \Xi \subset \Lambda$.
Hence, from properties (P1) and (P2) of the category,
\begin{equation}\label{categoria2}
\textrm{cat}_X (A_n',\Lambda) \leq \textrm{cat}_X (\Xi,\Lambda) \leq \textrm{cat}_X (\Xi,\Xi).
\end{equation}
The set $\Xi$ is clearly homeomorphic to $\mathbb{R}^3 \setminus B_R(0)$ (with $B_R(0)$ the open ball of radius $R$) and so
$\textrm{cat}_X (\Xi,\Xi) = 2$. Hence, \eqref{categoria1} and  \eqref{categoria2} yield
$$
\textrm{cat}_X (A_n,\Lambda) \leq 2,
$$
contradicting the fact that $A_n \in \mathcal{F}_3$.
\end{proof}

\noindent
From Lemmas \ref{pro:assumptionI}, \ref{pro:ps}, \ref{pro:compatti} and \ref{pro:livelli} we deduce that all the assumptions of Theorem \ref{thm:minmax} are satisfied and then Theorem \ref{teo:main} is proved.

\section{Applications}\label{sec4}

In this section, we give some applications of our main result. 

The first one deals with the motion of a charge under the effect of the electric and magnetic field generated by $N$ moving charges. 

For the second one, we move to the interpretation of equation \eqref{eq:lorentz} in relativistic celestial mechanics, dealing with the motion of a particle in a perturbed Kepler potential.

\subsection{The Li\'enard-Wiechert potentials}

Les us consider the motion of a charged particle 
with $m/q = 1$ under the effect of $N$ moving electric point charges. 

We denote by $q_1, \ldots, q_N$ the moving charges and by $r_1,\ldots ,r_N$ their trajectories, which we assume to be $C^2$ functions $r_j : \mathbb{R} \to \mathbb{R}^3$, $T$-periodic and such that $\vert \dot r_j(t) \vert < c$ for every $t \in [0,T]$ and $r_i(t) \neq r_j(t)$ for every $t \in [0,T]$ and $i \neq j$ (cf. Section \ref{sec3}). 
\bigskip
\noindent
Let us now set, for $i=1,\ldots,N$,
$$
\beta_i(t)=\dfrac{\dot{r_i}(t)}{c},\quad \forall \ t\in [0,T],\ i=1,\ldots, N,
$$
and observe that 
\begin{equation} \label{eq:betastima}
||\beta_i||_\infty<1.
\end{equation}
Moreover, we define $\eta_i:\Omega \to \RR^3$ by
$$
\eta_i (t,x)=\dfrac{x-r_i(t)}{|x-r_i(t)|},\quad \forall \ (t,x)\in \Omega
$$
and $t_i:\Omega \to \RR$ by the implicit relation
\begin{equation} \label{eq:temporitardato}
t_i=t-\dfrac{1}{c}\, |x-r_i(t_i)|.
\end{equation}
It is well-known that, for every $i=1,\ldots, N$, the number $t_i$ is the \textit{retarted} time. The existence and uniqueness of a solution of \eqref{eq:temporitardato} for a fixed $(t,x)\in \Omega$ is a standard fact in special relativity and it can be proved by means of a plain implicit function argument, which also implies that $t_i$ is a function of class $C^1$.
Moreover, the periodicity of $r_i$ implies that $t_i$ is $T$-periodic as a function of the time variable $t$.

The Li\'enard-Wiechert scalar and vector potentials generated by the point charge source $q_i$, $i=1, \ldots, N$, acting on 
a charge at the point $(t,x)$, are given, respectively, by
\begin{equation} \label{eq:defvi}
V_{i}(t,x)=\dfrac{q_i}{4\pi \varepsilon_0}\, \dfrac{1}{1-\eta_i(t_{i},x)\cdot \beta_i(t_i)}\, \dfrac{1}{\vert x-r_i(t_{i}) \vert}
\end{equation} 
and
\begin{equation} \label{eq:defAi}
A_{i}(t,x)=\dfrac{\beta_i(t_i)}{c}\, V_i(t,x),
\end{equation} 
where $t_i = t_i(t,x)$ and $\varepsilon_0$ is the vacuum permittivity. For future reference, let us recall that the corresponding electric and magnetic fields are given by
$$
\begin{array}{ll}
E_i(t,x)=&\displaystyle \dfrac{q_i}{4\pi \varepsilon_0}\, \left(\dfrac{\eta_i(t_i,x)-\beta_i(t_i)}{\gamma_i^2\, (1-\eta_i(t_{i},x)\cdot \beta_i(t_i))^3}\, \dfrac{1}{\vert x-r_i(t_{i}) \vert^2}\vphantom{\frac{\dot \beta_i(t)}{\beta(t)^3}}\right.\\
&\\
&\left. +
\dfrac{\eta_i(t_i,x)\times ((\eta_i(t_i,x)-\beta_i(t_i))\times \dot{\beta_i}(t_i))}{c\, (1-\eta_i(t_{i},x)\cdot \beta_i(t_i))^3}\, \dfrac{1}{\vert x-r_i(t_{i}) \vert}
\right),
\end{array}
$$
where $\gamma_i = 1/\sqrt{1-\vert \beta_i \vert^2}$ is the Lorentz factor, and
\begin{equation} \label{eq:magnetico}
B_i(t,x)=\dfrac{\eta_i(t_i,x)}{c}\times E_i(t,x),
\end{equation}
respectively (cf. \cite{Ja75}). 

Let us notice that $V_i$ and $A_i$ (and then $E_i$ and $B_i$) are well-defined in $\Omega$: indeed, 
from \eqref{eq:temporitardato} we first deduce that
\[
x-r_i(t_{i})=0 \quad \Longleftrightarrow \quad t=t_i,
\]
thus implying that $(t,x)=(t_i,r_i(t_i))$, which is impossible if $(t,x)\in \Omega$. On the other hand, if $(t,x)\in \Omega$ we have 
\[
|\eta_i(t_{i},x)\cdot \beta_i(t_i)|\leq ||\beta_i||_\infty
\]
and then, by \eqref{eq:betastima},
\begin{equation} \label{eq:stimadenominatore}
1-\eta_i(t_{i},x)\cdot \beta_i(t_i)\geq 1-||\beta_i||_\infty >0.
\end{equation}

We are now in a position to state our result on periodic motions under Li\'enard-Wiechert potentials.
\begin{Theorem} \label{teo:LW}
	In the above setting, let us assume that $q_i<0$, for every $i=1,\ldots,N$. Let
	\begin{equation} \label{eq:defVLW}
	V(t,x)=\sum_{i=1}^{N} V_i(t,x),\quad A(t,x)=\sum_{i=1}^N A_i(t,x),
	\end{equation}
	for every $(t,x)\in \Omega$, where $V_i$ and $A_i$, $i=1,\ldots, N$, are given in \eqref{eq:defvi} and \eqref{eq:defAi}, respectively. 

\noindent
Then, the corresponding Lorentz force equation \eqref{eq:lorentz} has infinitely many $T$-periodic solutions.
\end{Theorem}
 
\begin{proof} The result follows from Theorem \eqref{teo:main}. We need to show that $V$ and $A$ satisfy assumptions (V), (AV1) and (AV2).
	
\smallskip
As far as (V) is concerned, the assumption $q_i<0$, for every $i=1,\ldots, N$, implies that $V(t,x) < 0$, for every $(t,x)\in \Omega$.
Moreover, from \eqref{eq:temporitardato} and the definition of $\beta_i$ we deduce that
\begin{equation} \label{eq:pr11}
c(t-t_i)=|x-r_i(t_i)|\leq |x-r_i(t)|+|r_i(t)-r_i(t_i)|\leq |x-r_i(t)| + c ||\beta_i||_\infty (t-t_i),
\end{equation}
thus implying 
\[
c(t-t_i)\leq \dfrac{|x-r_i(t)|}{1-||\beta_i||_\infty }.
\]
Using this estimate in \eqref{eq:pr11}, we infer
$$
|x-r_i(t_i)|\leq \dfrac{1}{1-||\beta_i||_\infty}\, |x-r_i(t)|.
$$
This relation, together with \eqref{eq:stimadenominatore} and the sign assumption on the charges, implies that
\[
V_i(t,x) \leq \dfrac{q_i}{4\pi \varepsilon_0} \, \dfrac{1}{|x-r_i(t)|},
\]
for every $(t,x)\in \Omega$. Recalling that $V_j(t,x)<0$, for every $(t,x)\in \Omega$ and $j=1,\ldots, N$, we conclude that
\[
V(t,x) =V_i(t,x)+\sum_{j\neq i} V_j(t,x)\leq \dfrac{q_i}{4\pi \varepsilon_0} \, \dfrac{1}{|x-r_i(t)|},
\]
for every $(t,x)\in \Omega$. This proves the validity of \eqref{eq:Vsing} with $\kappa=\max \{\kappa_i:\, i=1,\ldots, N\}, \kappa_i=-q_i/4\pi \varepsilon_0$ and $\delta > 0$ arbitrary.

\smallskip
The validity of (AV1) is an immediate consequence of the definition of $A_i$ given in \eqref{eq:defAi}. Indeed, (AV1) is satisfied with
\[
\kappa'=\max\{||\beta_i||_\infty: \, i=1,\ldots, N\}
\] 
(observe that $\kappa'<1$ by \eqref{eq:betastima}).

\smallskip
Finally, we pass to the proof of the validity of (AV2), first observing that \eqref{eq-relelettricomagnetico} implies that (AV2) can be written as
\begin{equation} \label{eq-AV2riformulata}
\lim_{|x|\to +\infty} \left( |V(t,x)| + \vert \nabla_x E(t,x)\vert +\vert B(t,x) \vert \right) = 0,
	\end{equation}
	uniformly in $t\in \RR$. For every $i=1,\ldots, N$, from \eqref{eq:stimadenominatore} we infer that
\[
|V_i(t,x)| \leq \dfrac{-q_i}{4\pi \varepsilon_0} \, \dfrac{1}{1-||\beta_i||_\infty}\, \dfrac{1}{|x-r_i(t_i)|},
\]
for every $(t,x)\in \Omega$. Now, defining 
\[
\Theta=\max \{|r_i(t)|:\ t\in \RR,\ i=1,\ldots , N \},
\]
it is immediate to see that the set $E=\{(t,x)\in \RR\times \RR^3: \ |x|>\Theta+1\}$ satisfies $E\subset \Omega$ and that
\[
|V_i(t,x)| \leq \dfrac{-q_i}{4\pi \varepsilon_0} \, \dfrac{1}{1-||\beta_i||_\infty}\, \dfrac{1}{|x|-\Theta},\quad \forall \ (t,x)\in \Omega,\ |x|>\Theta+1,
\]
thus implying 
\begin{equation} \label{eq:stimainfinitoVi}
\lim_{|x|\to +\infty} |V_i(t,x)|=0,\quad \forall \ i=1,\ldots, N,
\end{equation} 
uniformly in $t\in \RR$. Taking again into account \eqref{eq:stimadenominatore}, the fact that $\eta_i$, $\beta_i$ and $\dot{\beta_i}$ are bounded and the definition of $\Theta$, we deduce that there exists $Z'>0$ such that
$$
|E_i(t,x)|\leq \dfrac{q_i}{4\pi \varepsilon_0} \, \dfrac{Z'}{(1-||\beta_i||_\infty)^3}\, \left(\dfrac{1}{|x|-\Theta}+\dfrac{1}{(|x|-\Theta)^2} \right),\quad \forall \ (t,x)\in \Omega, \ |x|>\Theta+1.
$$
This proves that
\begin{equation} \label{eq:stimainfinitoelettrico}
\lim_{|x|\to +\infty} |E_i(t,x)|=0, \quad \forall \ i=1,\ldots, N,
\end{equation}
uniformly in $t\in \RR$. Finally, from \eqref{eq:magnetico} and \eqref{eq:stimainfinitoelettrico}, recalling that $\eta_i$ is bounded, we infer
\begin{equation} \label{eq:stimainfinitomagnetico}
\lim_{|x|\to +\infty} |B_i(t,x)|=0, \quad \forall \ i=1,\ldots, N,
\end{equation}
uniformly in $t\in \RR$. From the fact that \eqref{eq:stimainfinitoVi},  \eqref{eq:stimainfinitoelettrico} and \eqref{eq:stimainfinitomagnetico} hold for every $i=1,\ldots, N$, recalling \eqref{eq:defVLW}, we can conclude that \eqref{eq-AV2riformulata} is fulfilled.
\end{proof}

\subsection{The forced relativistic Kepler problem}

Let us consider the equation
\begin{equation}\label{eqkep-fine}
	\frac{d}{dt}\left(\frac{m\dot{x}}{\sqrt{1-|\dot{x}|^{2}/c^{2}}}\right) =
	- \alpha \dfrac{x}{|x|^3}+\nabla_x U(t,x), \qquad x \in \mathbb{R}^3,
\end{equation}
interpreted as the relativistic Kepler problem ($m,\alpha > 0$), perturbed by an external force.

The following result holds true.
\begin{Theorem} \label{teo:MC}
	Let $U:\RR\times \RR^3\to \RR$ be a $C^1$ function, $T$-periodic in the first variable, satisfying $U(t,x) > 0$ for every $(t,x) \in \RR\times \RR^3$ and
	\begin{equation} \label{eq:ipoU}
	\lim_{|x|\to +\infty} (|U(t,x)|+|\nabla_x U(t,x)|)=0,
	\end{equation}
	uniformly in $t\in \RR$.
	
	\noindent
	Then, equation \eqref{eqkep-fine}
	 has infinitely many $T$-periodic solutions.
\end{Theorem}

\begin{proof} The result follows from Theorem \ref{teo:main}. Indeed, let us first observe that here 
$\Omega =\{(t,x)\in \RR\times \RR^3:\ x\neq 0\}$,
	\[
	V(t,x)=-\dfrac{\alpha}{m |x|}-\frac{1}{m}U(t,x),\quad A(t,x) = 0, \quad \forall \ (t,x)\in \Omega.
	\]	
	Then, from the sign condition on $U$ we plainly deduce that (V) (with $\kappa =1$ and arbitrary $\delta >0$) is satisfied.
	Moreover, assumption (AV1) is trivially fulfilled since $A\equiv 0$.
	Finally, from assumption \eqref{eq:ipoU}, we infer that
	\[
	\lim_{|x|\to +\infty} (|V(t,x)|+|\nabla_x V(t,x)|)=0,
	\]
uniformly in $t\in \RR$. Recalling again that $A\equiv 0$, this proves the validity of (AV2).	
\end{proof}

\noindent
A. Boscaggin\\
Dipartimento di Matematica ``Giuseppe Peano'', Università degli Studi di Torino\\
Via Carlo Alberto 10, 10123 Torino, Italy\\
alberto.boscaggin@unito.it
\medskip

\noindent
W. Dambrosio\\
Dipartimento di Matematica ``Giuseppe Peano'', Università degli Studi di Torino\\
Via Carlo Alberto 10, 10123 Torino, Italy\\
walter.dambrosio@unito.it
\medskip

\noindent
D. Papini \\
Dipartimento di Scienze Matematiche, Informatiche e Fisiche,
Università degli Studi di Udine\\
Via delle Scienze 206, 33100 Udine, Italy\\
\texttt{duccio.papini@uniud.it}

\end{document}